%% file: In_Domain_Malzer.tex
\documentclass[3p]{elsarticle}

\usepackage{lineno,hyperref}
\modulolinenumbers[5]

\usepackage{amsmath}
\usepackage{amssymb}
\usepackage{mathrsfs}
\usepackage{microtype}
\usepackage{subcaption}

\usepackage{pgfplots}
 \pgfplotsset{compat=newest}
  \usetikzlibrary{plotmarks}
  \usepackage{grffile}

\usepackage{tikz}
\usepackage{graphics,xcolor}
\definecolor{P285U}{cmyk}{0.89,0.43,0.0,0.0}
\definecolor{P285U_font}{cmyk}{0.89,0.43,0.0,0.4}
\definecolor{lgray}{cmyk}{0,0,0,0.2}
\definecolor{myblue}{cmyk}{100,75,0,0}
\definecolor{jkuBlue}{RGB}{4,110,152}
\definecolor{jkuBlue}{RGB}{0,120,170}
\definecolor{jkuCyan}{RGB}{100,180,190}
\definecolor{jkuYellow}{RGB}{230,195,35}
\definecolor{jkuGrey}{RGB}{125,130,140}
\definecolor{jkuDarkGrey}{RGB}{51,51,51}
\definecolor{jkuLightGreen}{RGB}{195,215,75}
\definecolor{jkuGreen}{RGB}{115,180,85}
\definecolor{jkuPurple}{RGB}{145,75,130}
\definecolor{jkuRed}{RGB}{205,90,80}

\usepackage{mdframed}

\newtheorem{thm}{Theorem}

\newtheorem{prop}[thm]{Proposition}

\newtheorem{exmp}{Example}
\newtheorem{rem}{Remark}
\newtheorem{proof}{Proof}

\graphicspath{{fig/}}

\begin{document}

\begin{frontmatter}

\title{On Structural Invariants in the Energy-Based Control of Infinite-Dimensional Port-Hamiltonian Systems with  In-Domain Actuation\tnoteref{mytitlenote}}
\tnotetext[mytitlenote]{This work has been supported by the Austrian Science Fund (FWF) under grant number P 29964-N32.}

\author[JKU]{Tobias Malzer\corref{mycorrespondingauthor}}
\ead{tobias.malzer_1@jku.at}

\author[BR]{Hubert Rams}
\ead{hubert.rams@br-automation.com}

\author[JKU]{Markus Sch\"{o}berl}
\ead{markus.schoeberl@jku.at}

\cortext[mycorrespondingauthor]{Corresponding author}

\address[JKU]{Institute of Automatic Control and Control Systems Technology, Johannes Kepler University Linz, Austria.}
\address[BR]{B\&R Industrial Automation GmbH, B\&R Stra\ss{}e 1, 5142 Eggelsberg, Austria.}

\begin{abstract}
This contribution deals with energy-based in-domain control of systems governed by partial differential equations with spatial domain up to dimension two. We exploit a port-Hamiltonian system description based on an underlying jet-bundle formalism, where we restrict ourselves to systems with $2$nd-order Hamiltonian. A certain power-conserving interconnection enables the application of a dynamic control law based on structural invariants. Furthermore, we use various examples such as beams and plates with in-domain actuation to demonstrate the capability of our approach.
\end{abstract}


\end{frontmatter}

\input{Text_In_Domain}

\bibliography{my_bib}

\end{document}

%% file: Text_In_Domain.tex
\section{Introduction}

For finite dimensional systems, the port-Hamiltonian (pH) framework
has proven as an appropriate system representation, as the structure
of the ordinary differential equations (ODEs) is related to the underlying
physics. From a control-engineering point of view, in particular the
fact that so-called power ports can be introduced plays an important
role, because it allows the application of energy-based control schemes,
see e.g. \cite{Schaft2000,Ortega2001}, where the objective is to
design a desired closed-loop behaviour by means of energy shaping
and damping injection.

A famous control scheme exploiting the occurence of power ports is
the well-known energy-Casimir method, which has already been extended
to systems governed by partial differential equations (PDEs). However,
in the infinite-dimensional scenario the generation of ports strongly
depends on the underlying approach, which unfortunately is -- in
contrast to the finite-dimensional scenario -- not unique. For control-engineering
purposes, in particular the so-called Stokes-Dirac scenario as well
as an approach based on an underlying jet-bundle structure -- where
the major difference of these approaches is the choice of the variables
-- have been established. For the well-known Stokes-Dirac scenario,
relying on the use of energy variables, boundary ports solely stem
from differential operators occurring in the system description, see
\cite{Schaft2002,Gorrec2005} for instance. In contrast, regarding
the jet-bundle approach, boundary ports basically result due to derivative
variables that may occur in the Hamiltonian, but can also be generated
by differential operators, see e.g. \cite{Ennsbrunner2005,Schoeberl2014a}.
While the Stokes-Dirac scenario exhibits a close relation to functional
analytic methods that can be used to address the well-posedness as
well as stability investigations of a problem, see \cite{Jacob2012},
for the systems considered in this contribution, that allow for a
variational characterisation, the jet-bundle approach is particularly
suitable. At this point it should be mentioned that the focus of this
paper is on exploiting geometric system properties, and thus, detailed
well-posedness and stability investigations based on functional analytic
methods are not presented. Therefore, we assume well-posedness and
confine ourselves to energy considerations regarding stability investigations.

The energy-Casimir method has proven to be an effective tool, where
for boundary-control systems with $1$-dimensional spatial domain
this method enables to derive finite-dimensional controllers, see,
e.g., \cite{Macchelli2004,Macchelli2017} for the controller design
in the Stokes-Dirac scenario and, e.g., \cite{Schoeberl2013a,Rams2017a}
for the jet-bundle approach. From a mathematical point of view, this
boundary-control scheme can be interpreted as the coupling of a PDE-
with an ODE-system at the actuated boundary of the plant. In general,
this methodology also allows the use of infinite-dimensional controllers,
see \cite{TrangVu2017,Trenchant2017}, corresponding to the coupling
of PDEs with PDEs, which would imply a further rise of complexity
regarding stability investigations.

In this paper, we focus on systems with in-domain actuation, where
we restrict ourselves to lumped inputs, which motivates the interconnection
of the plant to a finite-dimensional controller, representing the
coupling of a PDE- with an ODE-system within the spatial domain. In
particular, we aim to extend the foundings of \cite{Malzer2019},
where an in-domain control strategy for pH-systems with $2$nd order
Hamiltonian and $1$-dimensional spatial domain has been developed,
to systems with $2$-dimensional spatial domain. This approach relies
on a certain interconnection of the controller and the plant via its
distributed ports that can be deduced by means of a certain power-balance
relation. However, it should be noted that for the system class under
consideration the determination of this power-balance relation is
a non-trivial task, and thus, we employ an approach based on so-called
Cartan forms proposed in \cite{Schoeberl2018}. As an example for
this system class, we will study a plate modelled according to the
Kirchhoff-Love assumptions that is actuated by two pairs of piezoelectric
macro-fibre composite (MFC) patches; and we intend to develop an appropriate
control law. Therefore, the main contributions of this paper are as
follows: i) in Section \ref{sec:Infinite-Dimensional-PH-Systems},
which deals with the description of pH-systems for the infinite-dimensional
case, amongst others, a pH-system representation for a piezo-actuated
Kirchhoff-Love plate is derived. ii) to the best of our knowledge,
for the first time an energy-based in-domain control scheme being
able to cope with pH-systems with $2$-dimensional spatial domain
is proposed, see Subsection \ref{subsec:Structural_Invariants_1}.
iii) in Subsection \ref{subsec:structural_invariatns_2} it is shown
that the proposed control scheme is able to deal with systems restricted
to certain input assignments as well.

\section{Notation and Mathematical Preliminaries}

This paper is based on differential-geometric methods, where the notation
is similar to those of \cite{Saunders1989}. To keep the formulars
short and readable, we use tensor notation and apply Einsteins convention
on sums, where we do not indicate the range of the indices when the
are clear from the context. Furthermore, we will omit pullbacks in
order to avoid exaggerated notation. The standard symbols $\mathrm{d}$,
$\wedge$ and $\rfloor$ denote the exterior derivative, the exterior
wedge product and the Hook operator, allowing the natural contraction
between tensor fields, respectively. The set of all smooth functions
on a manifold $\mathcal{M}$ is denoted by $C^{\infty}(\mathcal{M})$.

This contribution deals with systems governed by PDEs, and therefore,
to be able to distinguish between dependent and independent variables,
we introduce a so-called bundle $\pi:\mathcal{E}\rightarrow\mathcal{B}$,
with $\left(z^{i}\right)$, $i=1,\ldots,m$, denoting the independent
coordinates of the base manifold $\mathcal{B}$ and $(z^{i},x^{\alpha})$,
where $x^{\alpha}$, with $\alpha=1,\ldots,n$, are the dependent
variables, those of the total manifold $\mathcal{E}$. Next, a section
of the bundle $\pi:\mathcal{E}\rightarrow\mathcal{B}$ is given by
the map $\phi:\mathcal{B}\rightarrow\mathcal{E}$, i.e. the dependent
and independent variables are related according to $x^{\alpha}=\phi^{\alpha}(z^{i})$.
Consequently, by introducing an ordered multi index $J=j_{1}\ldots j_{m}$
with $\sum j_{i}=k=\#J$, where an index $j_{i}$ indicates that the
derivative with respect to the independent variable $z^{i}$ is carried
out $j_{i}$ times, the $k$th-order spatial derivatives of a section
$\phi$ can be given by
\[
\frac{\partial^{k}}{(\partial z^{1})^{j_{1}}\cdots(\partial z^{m})^{j_{m}}}\phi^{\alpha}=\partial_{\left[J\right]}\phi^{\alpha}=\phi_{\left[J\right]}^{\alpha}\,.
\]
Furthermore, it should be noted that $1_{i}$ corresponds to a multi
index containing only zeros except the $i$th entry that is one, and
consequently, an increase of the $i$th entry of $J$ by one is given
by $J+1_{i}$. Now, in order to introduce so-called jet variables
or derivative coordinates, we consider further important geometric
objects, namely $r$th-order jet manifolds $\mathcal{J}^{r}(\mathcal{E})$
that are equipped with the coordinates $(z^{i},x^{\alpha},x_{\left[J\right]}^{\alpha})$,
where $x^{\alpha}=x_{\left[0\ldots0\right]}^{\alpha}$ holds.

Next, we introduce the tangent bundle $\tau_{\mathcal{E}}:\mathcal{T}(\mathcal{E})\rightarrow\mathcal{E}$,
which is equipped with the coordinates $(z^{i},x^{\alpha},\dot{z}^{i},\dot{x}^{\alpha})$
and the fibre basis $\partial_{i}=\partial/\partial z^{i}$, $\partial_{\alpha}=\partial/\partial x^{\alpha}$,
and hence, a vector field $v:\mathcal{E}\rightarrow\mathcal{T}(\mathcal{E})$
is a section given in local coordinates as $v=v^{i}\partial_{i}+v^{\alpha}\partial_{\alpha}$.
Moreover, the vertical tangent bundle $\nu:\mathcal{V}(\mathcal{E})\rightarrow\mathcal{E}$
possessing the coordinates $(z^{i},x^{\alpha},\dot{x}^{\alpha})$
is of particular interest since it allows to define a vertical vector
field $v=v^{\alpha}\partial_{\alpha}$ as a section of it. Consequently,
a vertical vector field $v$ prolonged to the $r$th-order jet manifold
$\mathcal{J}^{r}(\mathcal{E})$ reads as
\begin{equation}
j^{r}\left(v\right)=v+d_{\left[J\right]}\left(v^{\alpha}\right)\partial_{\alpha}^{\left[J\right]},\;d_{\left[J\right]}=\left(d_{\left[1_{1}\right]}\right)^{j_{1}}\ldots\left(d_{\left[1_{m}\right]}\right)^{j_{m}}\label{eq:prolongation}
\end{equation}
with $1\leq\#J\leq r$, where we exploit the total derivative $d_{\left[1_{i}\right]}$
with respect to the independent variable $z^{i}$, which is given
by
\[
d_{\left[1_{i}\right]}=\partial_{i}+x_{\left[J+1_{i}\right]}^{\alpha}\partial_{\alpha}^{\left[J\right]},\quad\partial_{\alpha}^{\left[J\right]}=\frac{\partial}{\partial x_{\left[J\right]}^{\alpha}}\,.
\]

A further important geometric object is the cotangent bundle $\tau_{\mathcal{E}}^{*}:\mathcal{T}^{*}\left(\mathcal{E}\right)\rightarrow\mathcal{E}$,
which is equipped with the coordinates $\left(z^{i},x^{\alpha},\dot{z}_{i},\dot{x}_{\alpha}\right)$
and the bases $\mathrm{d}z^{i}$ and $\mathrm{d}x^{\alpha}$. Thus,
a so-called $1$-form $\omega:\mathcal{E}\rightarrow\mathcal{T}^{*}(\mathcal{E})$
is a section, which reads as $\omega=\omega_{i}\mathrm{d}z^{i}+\omega_{\alpha}\mathrm{d}x^{\alpha}$
in local coordinates. By constructing certain pullback bundles, we
are able to address special densities -- that are quantities that
can be integrated -- of the form $\mathfrak{F}=\mathcal{F}\Omega$,
with $\mathcal{F}\in C^{\infty}\left(\mathcal{J}^{r}\left(\mathcal{E}\right)\right)$
implying that the coefficients may depend on derivative coordinates
as well. The corresponding integrated quantity $\mathscr{F}=\int_{\mathcal{B}}\mathcal{F}\Omega$
is called a functional. Here, we have used a volume element $\Omega$
that is defined on the base manifold $\mathcal{B}$, and consequently,
we have $\Omega=\mathrm{d}z^{1}\wedge\ldots\wedge\mathrm{d}z^{m}$
with $\textrm{dim}\left(\mathcal{B}\right)=m$ in local coordinates.
Furthermore, a boundary-volume form is denoted by $\Omega_{i}=\partial_{i}\rfloor\Omega$.
In this contribution, it is of particular interest to determine the
change of geometric objects along vector vields $v$, and therefore,
we exploit the so-called Lie derivative, which, exemplarily, reads
as $\text{L}_{v}\left(\omega\right)$ for a differential form $\omega$.

\section{Infinite-Dimensional PH-Systems\label{sec:Infinite-Dimensional-PH-Systems}}

In this section, an approach exploiting jet-bundle structures, see
e.g. \cite{Ennsbrunner2005,Schoeberl2008a}, is used to represent
infinite-dimensional pH-systems with $2$nd-order Hamiltonian density,
i.e. $\mathcal{H}\in C^{\infty}(\mathcal{J}^{2}(\mathcal{E}))$, actuated
within the ($1$- or $2$-dimensional) spatial domain. The approach
is based on a certain power-balance relation which can be used to
introduce power ports on the domain as well as on the boundary. Furthermore,
the section is completed by examples for systems with $1$- or $2$-dimensional
spatial domain.

First, we focus on systems with $2$-dimensional spatial domain, i.e.
we study Hamiltonian systems on the bundle $\pi:\mathcal{E}\rightarrow\mathcal{B}$
with $\left(z^{1},z^{2},x^{\alpha}\right)$ denoting the coordinates
of $\mathcal{E}$. The $2$nd-order Hamiltonian density is given by
$\mathfrak{H}=\mathcal{H}\Omega$ with $\mathcal{H}\in C^{\infty}(\mathcal{J}^{2}(\mathcal{E}))$,
where a volume element takes the local form $\Omega=\mathrm{d}z^{1}\wedge\mathrm{d}z^{2}$.
Now, we focus our interest on an evolutionary vector field $v=v^{\alpha}\partial_{\alpha}$,
corresponding to the set of PDEs
\begin{equation}
\dot{x}^{\alpha}=v^{\alpha}\,,\quad\text{with}\quad v^{\alpha}\in C^{\infty}(\mathcal{J}^{4}(\mathcal{E}))\,,\label{eq:evolution_equation}
\end{equation}
together with appropriate boundary conditions, where the time $t$
plays the role of the evolution parameter of the solution (well-posedness
provided). Next, the evolution of the Hamiltonian functional $\mathscr{H}=\int_{\mathcal{B}}\mathcal{H}\Omega$
along solutions of (\ref{eq:evolution_equation}) according to
\begin{equation}
\dot{\mathscr{H}}=\int_{\mathcal{B}}\text{L}_{j^{2}\left(v\right)}\left(\mathcal{H}\Omega\right)\label{eq:H_p_prolong_v}
\end{equation}
is of particular interest. Basically, by considering (\ref{eq:prolongation})
with $r=2$, the formal change $\dot{\mathscr{H}}$ can be deduced
by means of integration by parts. However, for the system class under
consideration, i.e. $2$nd-order Hamiltonian density and $2$-dimensional
spatial domain, the determination of the formal change is not straightforward
due to the ambiguity of the integration by parts which may yield wrong
boundary terms. To cope with that inconveniences, in \cite{Schoeberl2018}
an approach based on certain Cartan forms is proposed, where coordinates
adapted to the boundary as well as a boundary-volume form $\bar{\Omega}_{2}$
adapted to the boundary are used. Hence, based on \cite[Eqs. (13) and (14) ]{Schoeberl2018}
it is possible to derive the boundary operators\begin{subequations}\label{eq:boundary_operators}
\begin{align}
\delta^{\partial,1}\mathfrak{H} & =(\partial_{\alpha}^{\left[01\right]}\mathcal{H}-d_{\left[10\right]}(\partial_{\alpha}^{\left[11\right]}\mathcal{H})-d_{\left[01\right]}(\partial_{\alpha}^{\left[02\right]}\mathcal{H}))\mathrm{d}x^{\alpha}\wedge\bar{\Omega}_{2}\,,\label{eq:boundary_operator_1}\\
\delta^{\partial,2}\mathfrak{H} & =\partial_{\alpha}^{\left[02\right]}\mathcal{H}\mathrm{d}x_{\left[01\right]}^{\alpha}\wedge\bar{\Omega}_{2}\,,\label{eq:boundary_operator_2}
\end{align}
\end{subequations}whereas the variational derivative is given by
\begin{equation}
\delta\mathfrak{H}=\delta_{\alpha}\mathcal{H}\mathrm{d}x^{\alpha}\wedge\Omega\label{eq:domain_operator}
\end{equation}
with
\[
\delta_{\alpha}=\partial_{\alpha}-d_{\left[10\right]}\partial_{\alpha}^{\left[10\right]}-d_{[01]}\partial_{\alpha}^{\left[01\right]}+d_{[20]}\partial_{\alpha}^{\left[20\right]}+d_{[11]}\partial_{\alpha}^{\left[11\right]}+d_{[02]}\partial_{\alpha}^{\left[02\right]}\,.
\]
Furthermore, this approach allows to introduce a so-called decomposition
theorem given in \cite[Theorem 3.2]{Rams2018}, which plays a crucial
role not only for the determination of the formal change of the Hamiltonian,
but also for the derivation of structural invariants in Section \ref{sec:In-Domain-Control}.

\begin{thm}\cite[Theorem 3.2]{Rams2018}\label{thm:decomposition}
Let $\mathcal{H}\in C^{\infty}(\mathcal{J}^{2}(\mathcal{E}))$ be
a 2nd-order density and $v$ an evolutionary vector field. Then, the
integral $\int_{\mathcal{B}}\mathrm{L}_{j^{2}(v)}(\mathcal{H}\Omega)$
can be decomposed into
\[
\dot{\mathscr{H}}=\int_{\mathcal{B}}v\rfloor\delta\mathfrak{H}+\int_{\partial\mathcal{B}}v\rfloor\delta^{\partial,1}\mathfrak{H}+\int_{\partial\mathcal{B}}j^{1}\left(v\right)\rfloor\delta^{\partial,2}\mathfrak{H}
\]
with the domain operator (\ref{eq:domain_operator}), as well as both
the boundary operators according to (\ref{eq:boundary_operator_1})
and (\ref{eq:boundary_operator_2}).\end{thm} Next, we give a pH-system
representation making heavy use of a certain power-balance relation
that can be introduced based on Theorem \ref{thm:decomposition}.

The objective of the pH-system representation is to exploit the structure
of the governing evolution equations (\ref{eq:evolution_equation}),
which are therefore rewritten in the form\begin{subequations}\label{eq:pH_system_general}
\begin{align}
\dot{x} & =\left(\mathcal{J}-\mathcal{R}\right)\left(\delta\mathfrak{H}\right)+u\rfloor\mathcal{G}\,,\label{eq:pH_sys_dyn}\\
y & =\mathcal{G}^{*}\rfloor\delta\mathfrak{H}\,,\label{eq:pH_sys_output}
\end{align}
\end{subequations}where $\mathfrak{H}$ denotes a 2nd-order Hamiltonian.
Here, the (skew-symmetric) interconnection map $\mathcal{J}:\mathcal{T}^{*}\left(\mathcal{E}\right)\wedge\mathcal{T}^{*}\left(\mathcal{B}\right)\rightarrow\mathcal{V}\left(\mathcal{E}\right)$,
where $\mathcal{J}^{\alpha\beta}=-\mathcal{J}^{\alpha\beta}\in C^{\infty}(\mathcal{J}^{4}(\mathcal{E}))$
is met for the coefficients, describes the internal power flow of
the system, whereas the (symmetric and positive semi-definite) map
$\mathcal{R}:\mathcal{T}^{*}\left(\mathcal{E}\right)\wedge\mathcal{T}^{*}\left(\mathcal{B}\right)\rightarrow\mathcal{V}\left(\mathcal{E}\right)$,
satisfying $\mathcal{R}^{\alpha\beta}=\mathcal{R}^{\beta\alpha}\in C^{\infty}\left(\mathcal{J}^{4}\left(\mathcal{E}\right)\right)$
and $\left[\mathcal{R}^{\alpha\beta}\right]\geq0$ for the coefficient
matrix, is related to the dissipation effects. Furthermore, the input
map $\mathcal{G}:\mathcal{U}\rightarrow\mathcal{V}\left(\mathcal{E}\right)$
allows to include external inputs that may be distributed over (a
part of) the spatial domain -- i.e. both, the input coordinates $u^{\xi}\in\mathcal{U}$
as well as the coefficients $\mathcal{G}_{\xi}^{\alpha}$, may depend
(amongst others) on the spatial variables $z^{i}$ -- and is of great
relevance as we intend to develop in-domain control strategies in
this paper. Consequently, due to the distributed components $\mathcal{G}_{\xi}^{\alpha}$
of the adjoint output map $\mathcal{G}^{*}:\mathcal{T}^{*}\left(\mathcal{E}\right)\wedge\mathcal{T}^{*}\left(\mathcal{B}\right)\rightarrow\mathcal{Y}$,
(\ref{eq:pH_sys_output}) can be interpreted as distributed output
densities. Moreover, the fact that the input bundle $\rho:\mathcal{U}\rightarrow\mathcal{J}^{4}\left(\mathcal{E}\right)$
is dual to the output bundle $\varrho:\mathcal{Y}\rightarrow\mathcal{J}^{4}\left(\mathcal{E}\right)$,
see \cite[Section 4]{Ennsbrunner2005} or \cite[Section 3]{Schoeberl2008a},
yields the important relation
\begin{equation}
\left(u\rfloor\mathcal{G}\right)\rfloor\delta\mathfrak{H}=u\rfloor\left(\mathcal{G}^{*}\rfloor\delta\mathfrak{H}\right)=u\rfloor y\,,\label{eq:duality}
\end{equation}
which will play an important role for evaluating $\dot{\mathscr{H}}$.
Furthermore, in local coordinates (\ref{eq:pH_system_general}) reads
as\begin{subequations}\label{eq:pH_system_local_rep}
\begin{align}
\dot{x}^{\alpha} & =\left(\mathcal{J}^{\alpha\beta}-\mathcal{R}^{\alpha\beta}\right)\delta_{\beta}\mathcal{H}+\mathcal{G}_{\xi}^{\alpha}u^{\xi}\,,\\
y_{\xi} & =\mathcal{G}_{\xi}^{\alpha}\delta_{\alpha}\mathcal{H}\,,\label{eq:output_densities_plant}
\end{align}
\end{subequations}with $\alpha,\beta=1,\ldots,n$ and $\xi=1,\ldots,l$.

Next, it is of particular interest to reinterpret the formal change
of the Hamiltonian by keeping the pH-system representation (\ref{eq:pH_system_general})
in mind. By applying the decomposition Theorem \ref{thm:decomposition},
where we substitute $v=\dot{x}$ with (\ref{eq:pH_sys_dyn}) and use
the relation (\ref{eq:duality}), we conclude that
\[
\dot{\mathscr{H}}=-\int_{\mathcal{B}}\mathcal{R}\left(\delta\mathfrak{H}\right)\rfloor\delta\mathfrak{H}+\int_{\mathcal{\mathcal{B}}}u\rfloor y+\int_{\partial\mathcal{B}}\dot{x}\rfloor\delta^{\partial,1}\mathfrak{H}+\int_{\partial\mathcal{B}}\dot{x}_{\left[01\right]}\rfloor\delta^{\partial,2}\mathfrak{H}
\]
is divided into 4 parts. The energy of the system that is dissipated
-- e.g. due to damping -- is described by the expression $-\int_{\mathcal{B}}\mathcal{R}\left(\delta\mathfrak{H}\right)\rfloor\delta\mathfrak{H}$,
whereas the remaining terms denote collocation on the domain as well
as on the boundary. In particular the expression $\int_{\mathcal{\mathcal{B}}}u\rfloor y$,
which follows from the in- and outputs distributed over the spatial
domain, is of significant importance in this contribution. Moreover,
it should be noted that the 2 different boundary-port categories $\int_{\partial\mathcal{B}}\dot{x}\rfloor\delta^{\partial,1}\mathfrak{H}$
and $\int_{\partial\mathcal{B}}\dot{x}_{\left[01\right]}\rfloor\delta^{\partial,2}\mathfrak{H}$
are a consequence of the $2$nd-order Hamiltonian density. For the
sake of completeness, by using the local representation (\ref{eq:pH_system_local_rep}),
we can state the formal change in local coordinates according to
\[
\dot{\mathscr{H}}=-\int_{\mathcal{B}}\delta_{\alpha}\left(\mathcal{H}\right)\mathcal{R}^{\alpha\beta}\delta_{\beta}\left(\mathcal{H}\right)\Omega+\int_{\mathcal{B}}u^{\xi}y_{\xi}\Omega+\int_{\partial\mathcal{B}}\dot{x}^{\alpha}\delta_{\alpha}^{\partial,1}\mathcal{H}\bar{\Omega}_{2}+\int_{\partial\mathcal{B}}\dot{x}_{\left[01\right]}^{\alpha}\delta_{\alpha}^{\partial,2}\mathcal{H}\bar{\Omega}_{2}
\]
with the variational derivative (\ref{eq:domain_operator}) and the
boundary operators (\ref{eq:boundary_operator_1}) and (\ref{eq:boundary_operator_2}).
At this point, it should be stressed that we confine ourselves to
systems with in-domain actuation solely, implying that no power can
be extracted from or delivered to the system via the boundary, and
therefore, in this scenario the boundary ports $\int_{\partial\mathcal{B}}\dot{x}^{\alpha}\delta_{\alpha}^{\partial,1}\mathcal{H}\bar{\Omega}_{2}$
and $\int_{\partial\mathcal{B}}\dot{x}_{\left[01\right]}^{\alpha}\delta_{\alpha}^{\partial,2}\mathcal{H}\bar{\Omega}_{2}$
vanish identically.\begin{rem}\label{Rem:Boundary_operators}Although
boundary ports only play an tangential role in this contribution,
worth mentioning is the fact that the boundary terms can easily be
deduced by means of (\ref{eq:boundary_operator_1}) and (\ref{eq:boundary_operator_2}).
Furthermore, the boundary operators $\delta_{\alpha}^{\partial,1}$
and $\delta_{\alpha}^{\partial,2}$ are of major importance for the
determination of certain Casimir conditions for the controller design
treated in the next section.\end{rem}

Having discussed the framework for pH-systems with $2$nd-order Hamiltonian
and $2$-dimensional spatial domain, as an example a plate that is
modelled according to the Kirchhoff-Love theory and actuated by 2
pairs of piezoelectric MFC patches shall be studied, see Fig. \ref{fig:Schematic_Kirchhoff_Love_plate}.
\begin{figure}
\centering
\def \svgwidth{0.5\textwidth}
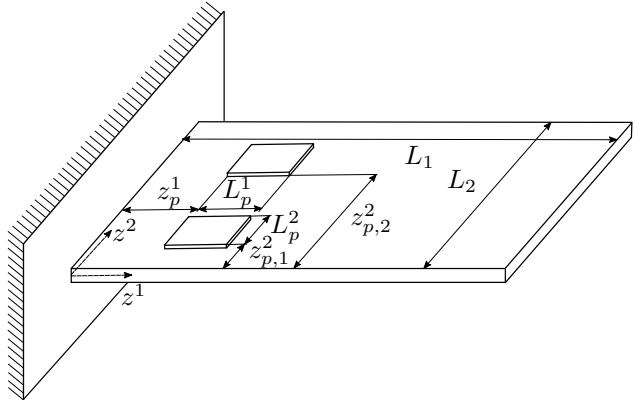\caption{\label{fig:Schematic_Kirchhoff_Love_plate}Schematic representation
of the piezo-actuated Kirchhoff-Love plate.}

\end{figure}
Hence, our intention is to find a pH-system representation for the
governing equation of motion -- that can be derived by using the
calculus of variation, see \cite{Meirovitch1997} for instance --
being useful regarding the energy-based controller design proposed
in Section \ref{sec:In-Domain-Control}.

\begin{exmp}[Piezo-actuated Kirchhoff-Love plate]\label{ex:Kirchhoff-Love_plate_piezo-actuated}Let
$\mathcal{B}=\{(z^{1},z^{2})|0\leq z^{1}\leq L_{1}\land0\leq z^{2}\leq L_{2}\}$
be the spatial domain of a rectangular plate modelled according to
the Kirchhoff-Love hypothesis, where the plate is clamped at the edge
$\partial\mathcal{B}_{1}=\{(z^{1},z^{2})|z^{1}=0\land0\leq z^{2}\leq L_{2}\}$,
while the remaining edges $\partial\mathcal{B}_{2}=\{(z^{1},z^{2})|0\leq z^{1}\leq L_{1}\land z^{2}=0\}$,
$\partial\mathcal{B}_{3}=\{(z^{1},z^{2})|z^{1}=L_{1}\land0\leq z^{2}\leq L_{2}\}$
and $\partial\mathcal{B}_{4}=\{(z^{1},z^{2})|0\leq z^{1}\leq L_{1}\land z^{2}=L_{2}\}$
are free. Moreover, the considered plate is actuated by 2 pairs of
piezoelectric patches, where each pair consists of 2 single patches
that are placed symmetrically on the upper and lower side of the plate.
Thus, the equation of motion for the system under consideration is
given by
\begin{multline}
\mu\left(z^{1},z^{2}\right)\ddot{w}=-d_{\left[20\right]}\left(\Xi\left(z^{1},z^{2}\right)\left(w_{\left[20\right]}+\nu w_{\left[02\right]}\right)\right)-d_{\left[02\right]}\left(\Xi\left(z^{1},z^{2}\right)\left(\nu w_{\left[20\right]}+w_{\left[02\right]}\right)\right)\\
-2d_{\left[11\right]}\left(\Xi\left(z^{1},z^{2}\right)\left(1-\nu\right)w_{\left[11\right]}\right)-\Lambda_{k}u_{in}^{k}\,,\label{eq:EOM_piezo_actuated_plate}
\end{multline}
with Poisson's ratio $\nu$ and $w_{\left[20\right]}=\frac{\partial^{2}w}{\partial(z^{1})^{2}}$
denoting the $2$nd-order derivative of the transversal plate deflection
$w$ with respect to $z^{1}$ for instance. Furthermore, the voltages
$u_{in}^{k}$ with $k=1,2$, which are applied to the MFC patches,
shall serve as manipulated variables, where $\Lambda_{k}=\Psi_{p}(a^{1}d_{\left[20\right]}\Gamma_{k}(z^{1},z^{2})+a^{2}d_{\left[02\right]}\Gamma_{k}(z^{1},z^{2}))$,
with $a^{1}$, $a^{2}$, $\Psi_{p}$ comprising several piezo-parameters,
states the spatial distribution of the inputs; see \cite{Meurer2012}
for a similar model where also dissipation effects are considered.
It is worth stressing that in (\ref{eq:EOM_piezo_actuated_plate})
the mass density $\mu(z^{1},z^{2})=\rho_{c}h_{c}+2\rho_{p}h_{p}(\Gamma_{1}(z^{1},z^{2})+\Gamma_{2}(z^{1},z^{2}))$,
with $\rho_{c}h_{c}$ and $\rho_{p}h_{p}$ denoting the mass densities
of the carrier layer and of the MFC paches, as well as the flexural
rigidity $\Xi(z^{1},z^{2})=E_{c}I_{c}+2\Xi_{p}(\Gamma_{1}(z^{1},z^{2})+\Gamma_{2}(z^{1},z^{2}))$,
with the flexural rigidity of the carrier layer $E_{c}I_{c}$ and
those of the MFC patches $\Xi_{p}$, are spatially dependent due to
the incorporation of the piezoelectric patches, which can be included
by the spatial characteristic functions
\begin{equation}
\Gamma_{k}\left(z^{1},z^{2}\right)=\left(h\left(z^{1}-z_{p}^{1}\right)-h\left(z^{1}-z_{p}^{1}-L_{p}^{1}\right)\right)\times\left(h\left(z^{2}-z_{p,k}^{2}\right)-h\left(z^{2}-z_{p,k}^{2}-L_{p}^{2}\right)\right)\,,\label{eq:charc_func_heavi_2D}
\end{equation}
with $k=1,2$, the heaviside function $h\left(\cdot\right)$ and the
geometric dimensions that are depicted in Fig. \ref{fig:Schematic_Kirchhoff_Love_plate}.
At this point it should be mentioned that in (\ref{eq:EOM_piezo_actuated_plate})
spatial derivatives of the characteristic functions $\Gamma_{k}(z^{1},z^{2})$
arise. Consequently, the use of the characteristic function (\ref{eq:charc_func_heavi_2D})
would require a weak formulation of the equation of motion. To be
able to exploit the strong formulation nonetheless, we approximate
the characteristic function (\ref{eq:charc_func_heavi_2D}) by the
spatially differentiable function
\begin{multline*}
\Gamma_{k}\left(z^{1},z^{2}\right)=\left(\frac{1}{2}\tanh\left(\sigma\left(z^{1}-z_{p}^{1}\right)\right)-\frac{1}{2}\tanh\left(\sigma\left(z^{1}-z_{p}^{1}-L_{p}^{1}\right)\right)\right)\times\\
\left(\frac{1}{2}\tanh\left(\sigma\left(z^{2}-z_{p,k}^{2}\right)\right)-\frac{1}{2}\tanh\left(\sigma\left(z^{2}-z_{p,k}^{2}-L_{p}^{2}\right)\right)\right)\,,
\end{multline*}
where $\sigma\in\mathbb{R}_{+}$ denotes a scaling factor. By virtue
of the plate configuration, see Fig. \ref{fig:Schematic_Kirchhoff_Love_plate},
regarding the boundary conditions we have\begin{subequations}\label{eq:BC_plate}
\begin{equation}
\left.\begin{array}{lcc}
w & = & 0\\
w_{\left[10\right]} & = & 0
\end{array}\right\} \quad\text{for}\quad\partial\mathcal{B}_{1}\,,
\end{equation}
whereas the shear force and the bending moment vanish along free edges,
i.e.
\begin{align}
\left.\begin{array}{ccc}
Q_{2} & = & 0\\
M_{2} & = & 0
\end{array}\right\} \quad\text{for}\quad & \partial\mathcal{B}_{3}\,,\\
\left.\begin{array}{ccc}
Q_{1} & = & 0\\
M_{1} & = & 0
\end{array}\right\} \quad\text{for}\quad & \partial\mathcal{B}_{2},\partial\mathcal{B}_{4}\,,
\end{align}
\end{subequations}with $Q_{1}=w_{\left[03\right]}+\left(2-\nu\right)w_{\left[21\right]}$,
$Q_{2}=-w_{\left[30\right]}-\left(2-\nu\right)w_{\left[12\right]}$,
$M_{1}=-w_{\left[02\right]}-\nu w_{\left[20\right]}$ and $M_{2}=w_{\left[20\right]}+\nu w_{\left[02\right]}$
denoting the shear force and the bending moment on the particular
edge.

Now, we focus our interest in finding a proper pH-system representation
for the piezo-actuated Kirchhoff-Love plate. To this end, we introduce
the momentum $p=\mu\left(z^{1},z^{2}\right)\dot{w}$, and consequently,
the total kinetic energy density of the underlying system reads as
\[
\mathcal{T}=\frac{1}{2\mu\left(z^{1},z^{2}\right)}p^{2}\,.
\]
If the plate is modelled based on the Kirchhoff-Love hypothesis --
i.e straight lines perpendicular to the midplane are supposed to remain
straight and perpendicular to the midplane during motion, and, the
transverse normal stress can be neglected as it is sufficient small
compared to the other normal stresses --, and it is assumed that
the piezoelectric material can be described by linear constitutive
relations, see \cite[Section 3.3.1]{Schroeck2011a}, the total potential
energy density follows to
\begin{multline*}
\mathcal{V}=\frac{1}{2}\Xi\left(z^{1},z^{2}\right)\left(\left(w_{\left[20\right]}\right)^{2}+\left(w_{\left[02\right]}\right)^{2}\right)+\frac{1}{2}\Xi\left(z^{1},z^{2}\right)\left(2\nu w_{\left[20\right]}w_{\left[02\right]}+2\left(1-\nu\right)\left(w_{\left[11\right]}\right)^{2}\right)\\
+\Psi_{p}\left(a^{1}w_{\left[20\right]}+a^{2}w_{\left[02\right]}\right)\Gamma_{k}\left(z^{1},z^{2}\right)u_{in}^{k}\,,
\end{multline*}
where the constants $\Xi_{p}$ and $\Psi_{p}$ comprise material parameters
of the MFC patches and $a^{1}$, $a^{2}$ stem from the linear constitutive
relations. Basically, with regard to boundary-control systems, the
total-energy density $\mathcal{T+\mathcal{V}}$ is used to obtain
a proper pH-system representation. However, due to the fact that we
focus on systems with in-domain actuation, we have included an input
part in (\ref{eq:pH_system_local_rep}). Therefore, the aim is to
find a Hamiltonian density such that an evaluation of $\delta_{w}\mathcal{H}$
yields the right-hand side of (\ref{eq:EOM_piezo_actuated_plate}),
but without the input part comprising $u_{in}^{k}$. To incorporate
the input part in the pH-system representation, we set the input-map
components to $g_{21}(z^{1},z^{2})=-\Lambda_{1}(z^{1},z^{2})$ and
$g_{22}(z^{1},z^{2})=-\Lambda_{2}(z^{1},z^{2})$ describing the spatial
distribution of the inputs $u_{in}^{k}$. Consequently, if we choose
the Hamiltonian density according to
\[
\mathcal{H}=\frac{1}{2\mu(z^{1},z^{2})}p^{2}+\frac{1}{2}\Xi(z^{1},z^{2})((w_{\left[20\right]})^{2}+(w_{\left[02\right]})^{2})+\frac{1}{2}\Xi(z^{1},z^{2})(2\nu w_{\left[20\right]}w_{\left[02\right]}+2(1-\nu)(w_{\left[11\right]})^{2}),
\]
a suitable pH-system description for the piezo-actuated Kirchhoff-Love
plate reads as\begin{subequations}\label{eq:pH_representation_piezoactuated_plate}
\begin{align}
\left[\begin{array}{c}
\dot{w}\\
\dot{p}
\end{array}\right]\! & =\!\left[\begin{array}{cc}
0 & 1\\
-1 & 0
\end{array}\right]\!\left[\begin{array}{c}
\delta_{w}\mathcal{H}\\
\delta_{p}\mathcal{H}
\end{array}\right]\!+\!\left[\begin{array}{cc}
0 & 0\\
g_{21} & g_{22}
\end{array}\right]\!\left[\begin{array}{c}
u_{in}^{1}\\
u_{in}^{2}
\end{array}\right],\\
\left[\begin{array}{c}
y_{1}\\
y_{2}
\end{array}\right] & \!=\!\left[\begin{array}{cc}
0 & g_{21}\\
0 & g_{22}
\end{array}\right]\!\left[\begin{array}{c}
\delta_{w}\mathcal{H}\\
\delta_{p}\mathcal{H}
\end{array}\right]\!=\!\left[\begin{array}{c}
g_{21}\dot{w}\\
g_{22}\dot{w}
\end{array}\right]\,,\label{eq:plate_output_densities}
\end{align}
\end{subequations}together with the boundary conditions (\ref{eq:BC_plate}),
where it should be stressed that $Q_{1}$, $Q_{2}$, $M_{1}$ and
$M_{2}$ can be deduced by evaluating (\ref{eq:boundary_operator_1})
and (\ref{eq:boundary_operator_2}), c.f. Rem. \ref{Rem:Boundary_operators}.
Next, to be able to introduce the power-ports for the system under
consideration, we determine the formal change of the Hamiltonian functional
$\mathscr{H}=\int_{\mathcal{B}}\mathcal{H}\Omega$, where the decomposition
Theorem \ref{thm:decomposition} shall be used. As we do not have
any boundary actuation, see the boundary conditions (\ref{eq:BC_plate}),
it becomes obvious that the boundary ports vanish. Furthermore, dissipation
effects are neglected in this example at all, and therefore, the formal
change reduces to
\begin{equation}
\dot{\mathscr{H}}=\int_{\mathcal{B}}\left(g_{21}(z^{1},z^{2})\dot{w}u_{in}^{1}+g_{22}(z^{1},z^{2})\dot{w}u_{in}^{2}\right)\Omega.\label{eq:h_p_plate}
\end{equation}
At this point it should be mentioned that a careful investigation
of the actuator parameters hidden in $\Psi_{p}$ shows that the unit
of the distributed output densities (\ref{eq:plate_output_densities})
is $\frac{\mathrm{A}}{\mathrm{m}^{2}}$. Hence, one can conclude that
the formal change (\ref{eq:h_p_plate}) corresponds to an electrical
power-balance relation.\end{exmp}By means of Ex. \ref{ex:Kirchhoff-Love_plate_piezo-actuated},
we have demonstrated that due to the incorporation of in-domain actuators,
which basically always exhibit a spatial distribution, power ports
that are distributed over (a part of) the spatial domain can arise.
In Section \ref{sec:In-Domain-Control}, these distributed power ports
shall be used for the controller design. However, from a control-engineering
point of view, it can also be of particular interest to investigate
distributed-parameter systems with actuators that can be modelled
-- at least approximately -- with an infinitesimal distribution,
where for the sake of simplicity we focus on distributed-parameter
systems with 1-dimensional spatial domain, i.e. we equip $\mathcal{B}$
with the independent coordinate $z^{1}$ solely. Consequently, a volume
form on $\mathcal{B}$ reads as $\Omega=\mathrm{d}z^{1}$ implying
that the corresponding boundary-volume form follows to $\Omega_{1}=\partial_{1}\rfloor\mathrm{d}z^{1}$.
For pH-systems with 1-dimensional spatial domain and 2nd-order Hamiltonian
density, the variational derivative in local coordinates is given
by
\[
\delta_{\alpha}=\partial_{\alpha}-d_{\left[1\right]}\partial_{\alpha}^{\left[1\right]}+d_{\left[2\right]}\partial_{\alpha}^{\left[2\right]}\,,
\]
whereas the boundary operators locally read as
\begin{equation}
\delta_{\alpha}^{\partial,1}\mathcal{H}=\partial_{\alpha}^{\left[1\right]}\mathcal{H}-d_{\left[1\right]}\left(\partial_{\alpha}^{\left[2\right]}\mathcal{H}\right),\quad\delta_{\alpha}^{\partial,2}\mathcal{H}=\partial_{\alpha}^{\left[2\right]}\mathcal{H}\,.\label{eq:boundary_operators_1D}
\end{equation}
In light of the aforementioned aspect, we introduce a specific form
of pH-systems according to\begin{subequations}\label{eq:Dirac_Klasse}
\begin{align}
\dot{x}^{\alpha} & =\left(\mathcal{J}^{\alpha\beta}-\mathcal{R}^{\alpha\beta}\right)\delta_{\beta}\mathcal{H}+\mathcal{G_{\xi}^{\alpha}}u^{\xi}\\
y_{\xi} & =\mathcal{G}_{\xi}^{\alpha}\delta_{\alpha}\mathcal{H}
\end{align}
with
\begin{equation}
\begin{array}{lclcccl}
\mathcal{G}_{\xi}^{\gamma} & = & 0 & \text{for} & \gamma & = & 1,\ldots,n_{1}\\
\mathcal{G}_{\xi}^{\rho} & = & \delta\left(z^{1}-A_{\xi}\right) & \textrm{for} & \rho & = & n_{1}+1,\ldots,n
\end{array}\label{eq:input_restrictions}
\end{equation}
\end{subequations}where $\delta\left(z^{1}-A_{\xi}\right)$ denotes
the Dirac delta function at the position $z^{1}=A_{\xi}$ indicating
that the inputs exhibit an infinitesimal spatial distribution. As
a consequence, the formal change of $\mathscr{H}$ follows to
\begin{equation}
\dot{\mathscr{H}}=-\int_{\mathcal{B}}\delta_{\alpha}\left(\mathcal{H}\right)\mathcal{R}^{\alpha\beta}\delta_{\beta}\left(\mathcal{H}\right)\Omega+u^{\xi}y_{\xi}+\left.\left(\dot{x}^{\alpha}\delta_{\alpha}^{\partial,1}\mathcal{H}\right)\right|_{\partial\mathcal{B}}+\left.\left(\dot{x}_{1}^{\alpha}\delta_{\alpha}^{\partial,2}\mathcal{H}\right)\right|_{\partial\mathcal{B}}\,,\label{eq:h_p_dirac}
\end{equation}
with the 0-dimensional boundary ports $\left.\left(\dot{x}^{\alpha}\delta_{\alpha}^{\partial,1}\mathcal{H}\right)\right|_{\partial\mathcal{B}}$
and $\left.\left(\dot{x}_{1}^{\alpha}\delta_{\alpha}^{\partial,2}\mathcal{H}\right)\right|_{\partial\mathcal{B}}$
that vanish again if systems with in-domain actuation solely are considered.
Consequently, it becomes obvious that we only have collocation located
pointwise on the domain, which is visualised by the following example.

\begin{exmp}[Pointwise actuated beam]\label{ex:Euler-Bernoulli_beam_pointwise-actuated}Now,
we consider an Euler-Bernoulli beam with the length $L$ actuated
at $z^{1}=A_{1}$ and $z^{1}=A_{2}$ by means of the forces $u^{1}$
and $u^{2}$, where the governing PDE is given by
\[
\rho A\ddot{w}=-EIw_{\left[4\right]}+\delta\left(z^{1}-A_{1}\right)u^{1}+\delta\left(z^{1}-A_{2}\right)u^{2}\,.
\]
Furthermore, both ends of the beam $\partial\mathcal{B}_{1}=0$ and
$\partial\mathcal{B}_{2}=L$ are free, and therefore, the boundary
conditions read as 
\[
\left.\begin{array}{ccc}
Q & = & 0\\
M & = & 0
\end{array}\right\} \quad\text{for}\quad\partial\mathcal{B}_{1},\partial\mathcal{B}_{2}
\]
with the shear force $Q=-EIw_{\left[3\right]}$ and the bending moment
$M=EIw_{\left[2\right]}$. Consequently, if we use the momentum $p=\rho A\dot{w}$
and the Hamiltonian density
\[
\mathcal{H}=\frac{1}{2\rho A}p^{2}+\frac{1}{2}EI\left(w_{\left[2\right]}\right)^{2}\,,
\]
the system under consideration can be written as
\begin{equation}
\left[\begin{array}{c}
\dot{w}\\
\dot{p}
\end{array}\right]=\left[\begin{array}{cc}
0 & 1\\
-1 & 0
\end{array}\right]\left[\begin{array}{c}
\delta_{w}\mathcal{H}\\
\delta_{p}\mathcal{H}
\end{array}\right]+\left[\begin{array}{cc}
0 & 0\\
\delta\left(z^{1}-A_{1}\right) & \delta\left(z^{1}-A_{2}\right)
\end{array}\right]\left[\begin{array}{c}
u^{1}\\
u^{2}
\end{array}\right],\label{eq:beam_lumped_inputs}
\end{equation}
where due to $y=\mathcal{G}^{*}\rfloor\delta\mathcal{H}$ the collocated
outputs corresponds to $y_{1}=\dot{w}|_{A_{1}}$ and $y_{2}=\dot{w}|_{A_{2}}$,
i.e. to velocities at defined positions. Thus, the formal change (\ref{eq:h_p_dirac})
follows to
\[
\dot{\mathscr{H}}=u^{1}\dot{w}|_{A_{1}}+u^{2}\dot{w}|_{A_{2}}\,,
\]
which states a mechanical power-balance relation.\end{exmp}

In this section, different classes of pH-systems, which also exhibit
different typs of power ports, have been considered. These power ports
shall be used for the proposed control by interconnection methodology
in the following.

\section{In-Domain Control using Structural Invariants\label{sec:In-Domain-Control}}

The aim of this section is to develop a control strategy based on
structural invariants being suitable for the different classes of
pH-systems treated in Section \ref{sec:Infinite-Dimensional-PH-Systems}.
These mentioned categories mainly differ in the dimension of the spatial
domain and the spatial distribution of the actuators; however, they
have in common that the inputs themselves are lumped, and consequently,
in light of this aspect, the application of a finite-dimensional controller
is motivated. Furthermore, in the Subsections \ref{subsec:Structural_Invariants_1}
and \ref{subsec:structural_invariatns_2} we derive casimir conditions
which differ due to the assumptions concerning the spatial distribution
of the actuators.

\subsection{Control by Interconnection}

Next, we adapt the control by interconnection strategy based on structural
invariants, which utilises damping injection and energy shaping in
order to stabilise certain equilibria, to in-domain actuated pH-system
with spatial domain up to dimension two. To achieve the damping-injection
part, the passivity of a pH-controller, coupled by a power-conserving
interconnection structure (PCIS) to the infinite-dimensional plant,
shall be exploited. Moreover, since the aim is to shape the energy
of the closed-loop system, we are interested in a relation between
the plant and the controller, which shall be obtained by means of
Casimir functionals.

Motivated by the lumped input of the considered plants, we use a finite-dimensional
pH-controller given in local coordinates as\begin{subequations}\label{eq:pH_controller}
\begin{align}
\dot{x}_{c}^{\alpha_{c}} & =\left(J_{c}^{\alpha_{c}\beta_{c}}-R_{c}^{\alpha_{c}\beta_{c}}\right)\partial_{\beta_{c}}H_{c}+G_{c,\xi}^{\alpha_{c}}u_{c}^{\xi}\,,\\
y_{c,\xi} & =G_{c,\xi}^{\beta_{c}}\partial_{\beta_{c}}H_{c}\,,\label{eq:controller_output}
\end{align}
\end{subequations}with $\alpha_{c},\beta_{c}=1,\ldots,n_{c}$. As
already mentioned, the main idea is to couple the finite-dimensional
controller to the infinite-dimensional plant in a power-conserving
manner, and therefore, we choose the dimension of the controller in-
and outputs according to $\mathrm{dim}\left(u_{c}\right)=\mathrm{dim}\left(y_{c}\right)=l$.
As the outputs of the plant may (in general) be distributed over (a
part of) the spatial domain, cf. Ex. \ref{ex:Kirchhoff-Love_plate_piezo-actuated},
to allow for a coupling the distributed output densities must be integrated
over $\mathcal{B}$, and therefore, a power-conserving interconnection
can be given by
\begin{equation}
u^{\xi}\int_{\mathcal{B}}y_{\xi}\Omega+u_{c}^{\xi}y_{c,\xi}=0\,.\label{eq:power_conserving_int}
\end{equation}
Here, it should be mentioned that we did not make any restriction
concerning the spatial dimension of the plant yet, i.e. $\mathrm{dim}\left(\mathcal{B}\right)=1,2$.
If we choose the feedback structure according to
\begin{equation}
u_{c}^{\xi}=K^{\xi\eta}\int_{\mathcal{B}}y_{\eta}\Omega\,,\qquad u^{\xi}=-K^{\xi\eta}y_{c,\eta}\,,\label{eq:PCIS_coupling}
\end{equation}
where $K^{\xi\eta}$ denotes the components of an appropriate map
$K$, a PCIS meeting (\ref{eq:power_conserving_int}) is obtained.
Furthermore, the closed loop, that results by using the coupling (\ref{eq:PCIS_coupling}),
is again a (mixed-dimensional) pH-system characterised by the Hamiltonian
$\mathscr{H}_{cl}=\mathscr{H}+H_{c}$. Next, by taking the coupling
(\ref{eq:PCIS_coupling}) into account -- and keeping in mind that
we consider systems with in-domain actuation solely --, a straightforward
calculation yields the formal change of $\mathscr{H}_{cl}$ according
to
\begin{equation}
\dot{\mathscr{H}}_{cl}=-\int_{\mathcal{B}}\delta_{\alpha}(\mathcal{H})\mathcal{R}^{\alpha\beta}\delta_{\beta}(\mathcal{H})\Omega-\partial_{\alpha_{c}}(H_{c})R_{c}^{\alpha_{c}\beta_{c}}\partial_{\beta_{c}}(H_{c})\leq0\,.\label{eq:hp_cl}
\end{equation}
Equ. (\ref{eq:hp_cl}) clearly highlights that we are able to inject
damping into the closed-loop system by means of the pH-controller
(\ref{eq:pH_controller}).\begin{rem}\label{Rem:Stability}It should
be stressed again that detailed stability investigations based on
functional-analytic methods are not in the scope of this contribution.
Instead, we focus on energy considerations, where $\dot{\mathscr{H}}_{cl}\leq0$
implies that the total energy is non-increasing along closed-loop
solutions (provided they exist). Hence, by using $\mathscr{H}_{cl}$
as Lyapunov candidate, the relations $\mathscr{H}_{cl}>0$ and $\dot{\mathscr{H}}_{cl}\leq0$
serve as necessary conditions for stability investigations in the
sense of Lyapunov.\end{rem}However, we are not content with damping
injection only; in particular, we additionally aim to shape the energy
of the closed-loop system. To this end, it is necessary to find a
relation between the plant and (some of) the controller states. Therefore,
in accordance with \cite{Schoeberl2013a,Rams2017a}, we are interested
in Casimir functionals of the form
\begin{equation}
\mathscr{C}^{\lambda}=x_{c}^{\lambda}+\int_{\mathcal{B}}\mathcal{C}^{\lambda}\Omega,\quad\mathcal{C}^{\lambda}\in C^{\infty}\left(\mathcal{J}^{2}\left(\mathcal{E}\right)\right),\label{eq:Casimir_functionals}
\end{equation}
with $\lambda=1,\ldots,\bar{n}\leq n_{c}$; however, it should be
stressed that in this contribution $\mathrm{dim}\left(\mathcal{B}\right)=1,2$
is valid. To serve as conserved quantity, the functionals (\ref{eq:Casimir_functionals})
have to fulfil $\dot{\mathscr{C}}^{\lambda}=0$ independently of $\mathcal{H}$
and $H_{c}$. Apart from that, the requirement $\dot{\mathscr{C}}^{\lambda}=0$
of course depends on the plant under consideration. Consequently,
in the following we distinguish between plants with actuators distributed
over (a part of) the spatial domain and plants with actuators modelled
with an infinitesimal distribution. In light of this aspect, we derive
different conditions for structural invariants depending on the particular
plant category and demonstrate the applicability of the proposed approach
by deriving controllers for the examples treated in Section \ref{sec:Infinite-Dimensional-PH-Systems}.

\subsection{Controller Scenario I\label{subsec:Structural_Invariants_1}}

This subsection deals with the controller design for infinite-dimensional
pH-systems with in-domain actuators that exhibit a spatial distribution.
Based on the findings of the previous subsection, in the following
proposition necessary conditions regarding the controller design for
the system class under consideration shall be given. 

\begin{prop}\label{prop:Casimir_cond_1}Let the interconnection
of the plant (\ref{eq:pH_system_local_rep}), where $\mathrm{dim}\left(\mathcal{B}\right)=2$
is valid, and the controller (\ref{eq:pH_controller}) be given by
(\ref{eq:PCIS_coupling}). Then, if the functionals (\ref{eq:Casimir_functionals})
meet the conditions\begin{subequations}\label{eq:Casimir_cond_1_}
\begin{align}
(J_{c}^{\lambda\beta_{c}}-R_{c}^{\lambda\beta_{c}}) & =0\label{eq:Casimir_cond_cont_map}\\
\delta_{\alpha}\mathcal{C}^{\lambda}(\mathcal{J}^{\alpha\beta}-\mathcal{R}^{\alpha\beta})+G_{c,\xi}^{\lambda}K^{\xi\eta}\mathcal{G}_{\eta}^{\beta} & =0\label{eq:Casimir_in_domain_condition}\\
\delta_{\alpha}\mathcal{C}^{\lambda}\mathcal{G}_{\xi}^{\alpha}K^{\xi\eta}G_{c,\eta}^{\alpha_{c}} & =0\label{eq:Casimir_cond_uncoup}\\
(\dot{x}^{\alpha}\delta_{\alpha}^{\partial,1}\mathcal{C}^{\lambda}+\dot{x}_{\left[01\right]}^{\alpha}\delta_{\alpha}^{\partial,2}\mathcal{C}^{\lambda})|_{\partial\mathcal{B}} & =0\label{eq:Casimir_boundary_condition}
\end{align}
\end{subequations} for $\lambda=1,\ldots,\bar{n}\leq n_{c}$, they
qualify as structural invariants of the closed loop.\end{prop}

\begin{proof}To prove the requirement $\dot{\mathscr{C}}^{\lambda}=0$,
we exploit the decomposition Theorem \ref{thm:decomposition} --
where we use $\mathcal{C}^{\lambda}\Omega$ instead of $\mathcal{H}\Omega$
now --, and consequently, the formal change of (\ref{eq:Casimir_functionals})
follows to
\[
\dot{\mathscr{C}}^{\lambda}=\dot{x}_{c}^{\lambda}+\int_{\mathcal{B}}\dot{x}^{\alpha}\delta_{\alpha}\mathcal{C}^{\lambda}\Omega+\int_{\partial\mathcal{B}}(\dot{x}^{\alpha}\delta_{\alpha}^{\partial,1}\mathcal{C}^{\lambda}+\dot{x}_{\left[01\right]}^{\alpha}\delta_{\alpha}^{\partial,2}\mathcal{C}^{\lambda})\bar{\Omega}_{2}\,.
\]
Then, by taking into account the dynamics of the plant (\ref{eq:pH_system_local_rep})
and the controller (\ref{eq:pH_controller}), as well as the coupling
(\ref{eq:PCIS_coupling}), we end up with
\begin{multline}
\dot{\mathscr{C}}^{\lambda}=(J_{c}^{\alpha_{c}\beta_{c}}-R_{c}^{\alpha_{c}\beta_{c}})\partial_{\beta_{c}}H_{c}+\int_{\mathcal{B}}(\delta_{\alpha}\mathcal{C}^{\lambda}(\mathcal{J}^{\alpha\beta}-\mathcal{R}^{\alpha\beta})+G_{c,\xi}^{\lambda}K^{\xi\eta}\mathcal{G}_{\eta}^{\beta})\delta_{\beta}\mathcal{H}\Omega\\
-\int_{\mathcal{B}}\delta_{\alpha}\mathcal{C}^{\lambda}\mathcal{G}_{\xi}^{\alpha}K^{\xi\eta}G_{c,\eta}^{\alpha_{c}}\partial_{\alpha_{c}}H_{c}\Omega+\int_{\partial\mathcal{B}}(\dot{x}^{\alpha}\delta_{\alpha}^{\partial,1}\mathcal{C}^{\lambda}+\dot{x}_{\left[01\right]}^{\alpha}\delta_{\alpha}^{\partial,2}\mathcal{C}^{\lambda})\bar{\Omega}_{2}=0\,,\label{eq:C_dot_subs}
\end{multline}
which yields exactly the conditions given in Prop. \ref{prop:Casimir_cond_1}.\end{proof}Now,
it is of interest to interpret the results of Prop. \ref{prop:Casimir_cond_1},
where we have the remarkable fact that -- in general -- the conditions
(\ref{eq:Casimir_cond_cont_map})--(\ref{eq:Casimir_cond_uncoup})
holds for systems with $1$- or $2$-dimensional spatial domain, cf.
\cite[Eqs. (21a)-(21c)]{Malzer2019}. Nevertheless, the differences
are hidden in the geometric objects and operators that of course strongly
depend on the dimension of the spatial domain. Furthermore, condition
(\ref{eq:Casimir_in_domain_condition}) -- and this is a major difference
compared to boundary-control schemes -- enables to relate controller
states with plant states within the spatial domain. However, condition
(\ref{eq:Casimir_boundary_condition}) -- making heavy use of the
boundary operators $\delta_{\alpha}^{\partial,1}$ and $\delta_{\alpha}^{\partial,2}$,
cf. Rem. \ref{Rem:Boundary_operators} -- describes the fact that
we are not able to find relations restricted to the boundary, which
is a consequence of the circumstance that systems with in-domain actuation
solely are considered. Next, the applicability of the proposed control
strategy shall be demonstrated by developing a Casimir-controller
for the piezo-actuated Kirchhoff-Love plate.

\begin{exmp}[Energy-Casimir controller for Ex. 1]\label{ex:Casimir_controller_plate}Now,
we intend to exploit the pH-system representation of the piezo-actuated
Kirchhoff-Love plate given in (\ref{eq:pH_representation_piezoactuated_plate})
in order to derive an energy-based control law. The aim is to move
the plate from the initial position $w_{0}(z^{1},z^{2})=0$ to the
special rest postion
\begin{equation}
w^{d}\!=\!\left\{ \!\begin{array}{ll}
a(z^{1})^{2}k(z^{2}) & \!\text{for}\;0\leq z^{1}<z_{b}^{1}\\
(b(z^{1}\!-\!z_{b}^{1})\!+\!a(z_{b}^{1})^{2})k(z^{2})\! & \!\text{for}\;z_{b}^{1}\leq z^{1}\leq L^{1}
\end{array}\right.\label{eq:desired_equilibrium_plate}
\end{equation}
with $k(z^{2})=-c+dz^{2}$, $z_{b}^{1}=z_{p}^{1}+L_{p}^{1}$ and $a,b,c,d\in\mathbb{R}$.
To this end, by considering the dimension of the output densities
(\ref{eq:plate_output_densities}), 2 controller states shall be related
to the plant. To fulfil the conditions (\ref{eq:Casimir_cond_1_}),
we choose $\mathcal{C}^{1}=-g_{21}(z^{1},z^{2})w$ and $\mathcal{C}^{2}=-g_{22}(z^{1},z^{2})w$
fixing a part of the controller mappings $J_{c}$, $R_{c}$ and $G_{c}$
because we set $K^{\xi\eta}=\delta^{\xi\eta}$ with the Kronecker-Delta
symbol meeting $\delta^{\xi\eta}=1$ for $\xi=\eta$ and $\delta^{\xi\eta}=0$
for $\xi\neq\eta$. Furthermore, this ansatz allows for a relation
between the plant and the first $2$ controller states as it yields\begin{subequations}\label{eq:casimirs_plate}
\begin{align}
x_{c}^{1} & =\int_{\mathcal{B}}g_{21}(z^{1},z^{2})w\Omega\,,\\
x_{c}^{2} & =\int_{\mathcal{B}}g_{22}(z^{1},z^{2})w\Omega\,,
\end{align}
\end{subequations}by choosing appropriate initial states for the
controller. Compared to boundary-control schemes, the relations (\ref{eq:casimirs_plate})
are a major difference as the controller states $x_{c}^{1}$ and $x_{c}^{2}$
corresponds to a plant state that is weighted and integrated over
the 2-dimensional spatial domain, whereas boundary controller exploit
plant states restricted to the actuated boundary. Note that we have
not determined the dimension of the controller, which can be interpreted
as degree of freedom, yet. In this regard, damping shall be injected
into the closed-loop system by means of $2$ further controller states,
and therefore, we set $n_{c}=4$. Keeping the preceding facts in mind,
we find that the controller dynamics are restricted to the mappings\begin{subequations}\label{eq:Cont_Dyn_plate}
\begin{align}
J_{c}-R_{c} & =\left[\begin{array}{cccc}
0 & 0 & 0 & 0\\
0 & 0 & 0 & 0\\
0 & 0 & -R_{c}^{33} & J_{c}^{34}-R_{c}^{34}\\
0 & 0 & -J_{c}^{34}-R_{c}^{34} & -R_{c}^{44}
\end{array}\right]\,,\\
G_{c} & =\left[\begin{array}{cccc}
1 & 0 & G_{c,1}^{3} & G_{c,1}^{4}\\
0 & 1 & G_{c,2}^{3} & G_{c,2}^{4}
\end{array}\right]^{T}\,.
\end{align}
\end{subequations}Next, it remains to assign the Hamiltonian of the
controller. Here, it should be mentioned that the equilibrium (\ref{eq:desired_equilibrium_plate})
is one that requires non-zero power, and therefore, we have to include
an appropriate term in the controller Hamiltonian. Furthermore, we
intend to obtain a minimum of the closed-loop Hamiltonian $\mathscr{H}_{cl}=\int_{\mathcal{B}}\mathcal{H}\Omega+H_{c}$
that involves the desired equilibrium (\ref{eq:desired_equilibrium_plate}).
To this end, we exploit the relations
\begin{align*}
x_{c}^{1,d} & =\int_{\mathcal{B}}g_{21}\left(z^{1},z^{2}\right)w^{d}\Omega\,,\\
x_{c}^{2,d} & =\int_{\mathcal{B}}g_{22}\left(z^{1},z^{2}\right)w^{d}\Omega\,,
\end{align*}
which are a consequence of (\ref{eq:casimirs_plate}), and choose
\[
H_{c}=\frac{c_{1}}{2}(x_{c}^{1}-x_{c}^{1,d}-\frac{u_{s}^{1}}{c_{1}})^{2}+\frac{c_{2}}{2}(x_{c}^{2}-x_{c}^{2,d}-\frac{u_{s}^{2}}{c_{2}})^{2}+\frac{1}{2}M_{c,\mu_{c}\nu_{c}}x_{c}^{\mu_{c}}x_{c}^{\nu_{c}},
\]
with the positive definite matrix $\left[M_{c}\right]$, $M_{c,\mu_{c}\nu_{c}}\in\mathbb{R}$
for $\mu_{c},\nu_{c}=3,4$ and the positive constants $c_{1},c_{2}>0$.
As already mentioned, we have chosen $K^{\xi\eta}=\delta^{\xi\eta}$,
which yields the PCIS
\[
\begin{array}{cccccc}
u_{c}^{1} & = & \int_{\mathcal{B}}g_{21}(z^{1},z^{2})\dot{w}\Omega,\quad & u^{1} & = & y_{c,1},\\
u_{c}^{2} & = & \int_{\mathcal{B}}g_{22}(z^{1},z^{2})\dot{w}\Omega,\quad & u^{2} & = & y_{c,2},
\end{array}
\]
and consequently, the formal change of $\mathscr{H}_{cl}$ follows
to
\[
\dot{\mathscr{H}}_{cl}=-x_{c}^{\mu_{c}}M_{c,\mu_{c}\nu_{c}}R_{c}^{\nu_{c}\rho_{c}}M_{c,\rho_{c}\vartheta_{c}}x_{c}^{\vartheta_{c}}\leq0,
\]
with $\rho_{c},\vartheta_{c}=3,4$. As we do not carry out extensive
stability investigations, cf. Rem. \ref{Rem:Stability}, the simulation
results given in the Figs. \ref{fig:final_state_plate} and \ref{fig:upper_boundary}
are used to verify the applicability of the proposed approach. In
Fig. \ref{fig:final_state_plate}, the final plate deflection $w$
is depicted over the spatial domain $(z^{1},z^{2})$. 
\begin{figure}

\centering
\input{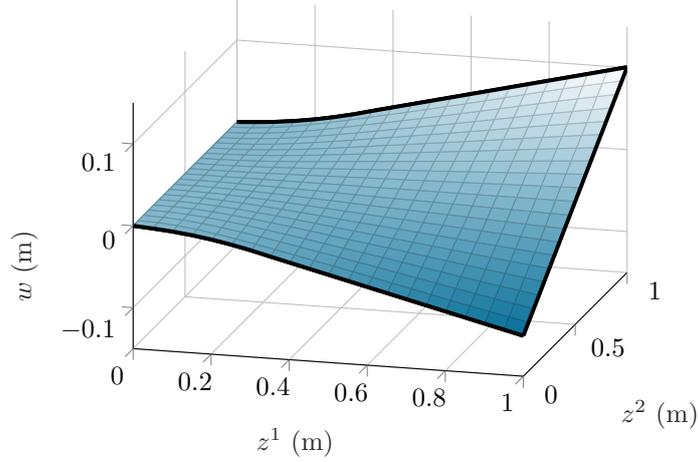}\caption{\label{fig:final_state_plate}Simulation result for the final plate
deflection $w$ over the spatial domain $(z^{1},z^{2})$.}

\end{figure}
Here, for the sake of simplicity, all plate parameters are set to
$1$, expect for the Poisson's ratio $\nu=0.2$. Furthermore, regarding
the desired equilibrium (\ref{eq:desired_equilibrium_plate}) we have
chosen $a=0.16$, $b=0.12$, $c=1$ and $d=2$. The MFC patches with
$L_{p}^{1}=L_{p}^{2}=0.25$ are placed at $z_{p}^{1}=0.25$, $z_{p,1}^{2}=0.1$
and $z_{p,2}^{1}=0.65$, see Fig. \ref{fig:Schematic_Kirchhoff_Love_plate}.
The controller parameters have been chosen as $J_{c}^{34}=1$, $R_{c}^{34}=-1$,
$R_{c}^{33}=200$, $R_{c}^{44}=150$, $M_{c}^{33}=M_{c}^{44}=10^{4}$,
$M_{c}^{34}=0$, $G_{c}^{31}=G_{c}^{41}=100$, $G_{c}^{32}=G_{c}^{42}=0$
and $c_{1}=c_{2}=0.1$. Worth stressing is the fact that the finite
difference-coefficient method has been applied as discretisation scheme,
where each direction of the plate have been divided into $20$ intervals.
\begin{figure}

\centering\input{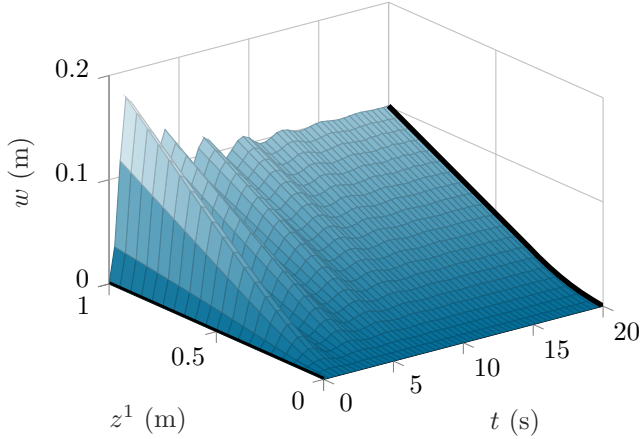}\caption{\label{fig:upper_boundary}Simulation results for the deflection $w$
of the edge $\partial\mathcal{B}_{4}$ over time $t$ and spatial
coordinate $z^{1}$.}

\end{figure}
\end{exmp} Having demonstrated the capability of the proposed control
scheme for pH-systems with 2-dimensional spatial domain, in the following
subsection a dynamic controller for pH-systems, actuated pointwise
within the ($1$-dimensional) spatial domain, is derived.

\subsection{Controller Scenario II\label{subsec:structural_invariatns_2}}

In this subsection, we restrict ourselves to systems described by
(\ref{eq:Dirac_Klasse}), which has the consequence that modified
conditions, being suitable for the system class under consideration,
can be deduced. Furthermore, an energy-based controller stabilising
a certain rest position shall be derived, see \cite[Subsection 3.2.5]{Rams2018}.

\begin{prop}\label{Prop:Casimir_cond_2}For the closed-loop system
that is obtained by an interconnection of the certain system class
(\ref{eq:Dirac_Klasse}) and the controller (\ref{eq:pH_controller})
via the coupling (\ref{eq:PCIS_coupling}), the functionals (\ref{eq:Casimir_functionals})
have to meet the conditions\begin{subequations}\label{eq:Casimir_conditions_2_}
\begin{align}
(J_{c}^{\lambda\beta_{c}}-R_{c}^{\lambda\beta_{c}}) & =0\\
\delta_{\alpha}\mathcal{C}^{\lambda}(\mathcal{J}^{\alpha\gamma}-\mathcal{R}^{\alpha\gamma}) & =0\label{eq:Casimir_cond_2_unact_states}\\
\delta_{\alpha}\mathcal{C}^{\lambda}(\mathcal{J}^{\alpha\rho}-\mathcal{R}^{\alpha\rho})+(G_{c,\xi}^{\lambda}K^{\xi\rho})|_{A_{\xi}^{i}} & =0\label{eq:Casimir_2_in_domain_condition}\\
\delta_{\alpha}\mathcal{C}^{\lambda}\mathcal{G}_{\xi}^{\alpha}K^{\xi\eta}G_{c,\eta}^{\alpha_{c}} & =0\\
(\dot{x}^{\alpha}\delta_{\alpha}^{\partial,1}\mathcal{C}^{\lambda}+\dot{x}_{\left[1\right]}^{\alpha}\delta_{\alpha}^{\partial,2}\mathcal{C}^{\lambda})|_{\partial\mathcal{B}} & =0
\end{align}
\end{subequations}for $\gamma=1,\ldots,n_{1}$ and $\rho=n_{1}+1,\ldots,n$
to qualify as structural invariants.\end{prop}

\begin{proof}The proof follows the intention of those of Prop. \ref{prop:Casimir_cond_1},
and consequently, we have
\begin{multline}
\dot{\mathscr{C}}^{\lambda}=(J_{c}^{\alpha_{c}\beta_{c}}-R_{c}^{\alpha_{c}\beta_{c}})\partial_{\beta_{c}}H_{c}+\int_{\mathcal{B}}(\delta_{\alpha}\mathcal{C}^{\lambda}(\mathcal{J}^{\alpha\beta}-\mathcal{R}^{\alpha\beta})+G_{c,\xi}^{\lambda}K^{\xi\eta}\mathcal{G}_{\eta}^{\beta})\delta_{\beta}\mathcal{H}\Omega\\
-\int_{\mathcal{B}}\delta_{\alpha}\mathcal{C}^{\lambda}\mathcal{G}_{\xi}^{\alpha}K^{\xi\eta}G_{c,\eta}^{\alpha_{c}}\partial_{\alpha_{c}}H_{c}\Omega+(\dot{x}^{\alpha}\delta_{\alpha}^{\partial,1}\mathcal{C}^{\lambda}+\dot{x}_{\left[1\right]}^{\alpha}\delta_{\alpha}^{\partial,2}\mathcal{C}^{\lambda})|_{\partial\mathcal{B}}\,,\label{eq:C_dot_subs-1}
\end{multline}
now. Consequently, by substituting the specific input restrictions
(\ref{eq:input_restrictions}) in (\ref{eq:C_dot_subs-1}), the proof
follows immediately.\end{proof}

Next, it remains to draw conclusions to the findings of Prop. \ref{prop:Casimir_cond_1}.
Similar to (\ref{eq:Casimir_in_domain_condition}), the condition
(\ref{eq:Casimir_2_in_domain_condition}) allows for a relation between
the controller and the plant that is now restricted to a certain position
of the spatial domain due to the specific input assignment we made.
Furthermore, this assignment implies condition (\ref{eq:Casimir_cond_2_unact_states})
for the system states where no input is acting.

With the preceding findings in mind, the pointwise actuated beam of
Ex. \ref{ex:Euler-Bernoulli_beam_pointwise-actuated} shall be used
to demonstrate the proposed approach, see also \cite[Subsection 3.2.5]{Rams2018}.

\begin{exmp}[Energy-Casimir controller for Ex. 2]Now, a controller
that stabilises the desired equilibrium 
\begin{equation}
w^{d}=az^{1}+b,\quad w_{\left[1\right]}^{d}=a\label{eq:equilib_dirac_beam}
\end{equation}
for the pointwise actuated beam of Ex. \ref{ex:Euler-Bernoulli_beam_pointwise-actuated}
shall be derived. There, we consider actuators and sensors with infinitesimal
distribution, and for the particular example the outputs are velocities
at defined positions. As a consequence, the coupling (\ref{eq:PCIS_coupling})
reduces to $u_{c}^{\xi}=\delta^{\xi\eta}y_{\eta}$ and $u^{\xi}=-\delta^{\xi\eta}y_{c,\eta}$
with $\delta^{\xi\eta}$ denoting the Kronecker-Delta symbol for $\xi,\eta=1,2$.
Moreover, due to the fact that two pointwise outputs are present,
we aim to relate $2$ controller states to the plant. To be able to
inject damping into the closed-loop system, we choose the controller
dimension to $n_{c}=4$. With regard to our control objectives, we
set $\mathcal{C}^{1}=-\delta(z^{1}-A_{1})w$ and $\mathcal{C}^{2}=-\delta(z^{1}-A_{2})w$,
where straightforward calculations show that they satisfy the conditions
(\ref{eq:Casimir_conditions_2_}) and yield the important relations
$x_{c}^{1}=w|_{A_{1}}$, $x_{c}^{2}=w|_{A_{2}}$, implying the remarkable
fact that we have the same structure for the controller dynamics as
in Ex. \ref{ex:Casimir_controller_plate}, see (\ref{eq:Cont_Dyn_plate}),
even though the problem is quite different. If we consider $x_{c}^{1,d}=w^{d}|_{A_{1}}=aA_{1}+b$,
$x_{c}^{2,d}=w^{d}|_{A_{2}}=aA_{2}+b$ and choose
\[
H_{c}=\frac{c_{1}}{2}(x_{c}^{1}-x_{c}^{1,d})^{2}+\frac{c_{2}}{2}(x_{c}^{2}-x_{c}^{2,d})^{2}+\frac{1}{2}M_{c,\mu_{c}\nu_{c}}x_{c}^{\mu_{c}}x_{c}^{\nu_{c}},
\]
with the positive definite matrix $\left[M_{c}\right]$, $M_{c,\mu_{c},\nu_{c}}\in\mathbb{R}$
for $\mu_{c},\nu_{c}=3,4$, and the positive constants $c_{1},c_{2}>0$,
the equilibrium (\ref{eq:equilib_dirac_beam}) becomes a part of the
minimum of
\begin{equation}
\mathscr{H}_{cl}=\int_{\mathcal{B}}(\frac{1}{2}EI(w_{\left[2\right]})^{2}+\frac{1}{2\rho A}p^{2})\Omega+H_{c}.\label{eq:H_cl_dirac}
\end{equation}
The positive definiteness of (\ref{eq:H_cl_dirac}) together with
\[
\dot{\mathscr{H}}_{cl}=-x_{c}^{\mu_{c}}M_{c,\mu_{c}\nu_{c}}R_{c}^{\nu_{c}\rho_{c}}M_{c,\rho_{c}\vartheta_{c}}x_{c}^{\vartheta_{c}}\leq0
\]
for $\mu_{c},\nu_{c}=3,4$, yield necessary conditions for the stability
of the desired equilibrium (\ref{eq:equilib_dirac_beam}), cf. Rem
\ref{Rem:Stability}.\end{exmp}

\section{Summary and Outlook}

In this paper, a control methodology based on structural invariants,
that is able to cope with in-domain actuated pH-systems with spatial
domain up to dimension two, has been presented. We restricted ourselves
to the scenario of lumped inputs and exploited a certain PCIS to deal
with the distributed output densities that (may) arise due to the
spatial distribution of the actuators. Furthermore, as discussed in
Ex. \ref{ex:Casimir_controller_plate}, as discretisation scheme the
finite difference-quotient method has been applied as it allows to
easily include in-domain inputs. However, this discretisation method
also has some drawbacks like the quadratically rising complexity,
and therefore, in future investigations we shall adapt more sophisticated
-- like e.g. structure preserving -- discretisation schemes for
spatially higher dimensional systems with in-domain actuation to our
framework.

%% file: Schematic_Piezo_Platte.pdf_tex
\begingroup%
  \makeatletter%
  \providecommand\color[2][]{%
    \errmessage{(Inkscape) Color is used for the text in Inkscape, but the package 'color.sty' is not loaded}%
    \renewcommand\color[2][]{}%
  }%
  \providecommand\transparent[1]{%
    \errmessage{(Inkscape) Transparency is used (non-zero) for the text in Inkscape, but the package 'transparent.sty' is not loaded}%
    \renewcommand\transparent[1]{}%
  }%
  \providecommand\rotatebox[2]{#2}%
  \ifx\svgwidth\undefined%
    \setlength{\unitlength}{795.85008248bp}%
    \ifx\svgscale\undefined%
      \relax%
    \else%
      \setlength{\unitlength}{\unitlength * \real{\svgscale}}%
    \fi%
  \else%
    \setlength{\unitlength}{\svgwidth}%
  \fi%
  \global\let\svgwidth\undefined%
  \global\let\svgscale\undefined%
  \makeatother%
  \begin{picture}(1,0.64262241)%
    \put(0,0){\includegraphics[width=\unitlength,page=1]{Schematic_Piezo_Platte.pdf}}%
    \put(0.25154645,0.34179731){\color[rgb]{0,0,0}\makebox(0,0)[lb]{\smash{ }}}%
    \put(0,0){\includegraphics[width=\unitlength,page=2]{Schematic_Piezo_Platte.pdf}}%
    \put(0.63537434,0.38051573){\color[rgb]{0,0,0}\makebox(0,0)[lb]{\smash{$L_1$}}}%
    \put(0.70320129,0.34148557){\color[rgb]{0,0,0}\makebox(0,0)[lb]{\smash{$L_2$}}}%
    \put(0.18129131,0.15299829){\color[rgb]{0,0,0}\makebox(0,0)[lb]{\smash{$z^{1}$}}}%
    \put(0.16774407,0.25683362){\color[rgb]{0,0,0}\makebox(0,0)[lb]{\smash{$z^{2}$}}}%
    \put(0.54878963,0.28185238){\color[rgb]{0,0,0}\makebox(0,0)[lb]{\smash{$z_{p,2}^{2}$}}}%
    \put(0.38588397,0.22945606){\color[rgb]{0,0,0}\makebox(0,0)[lb]{\smash{$z_{p,1}^{2}$}}}%
    \put(0.23901176,0.32510046){\color[rgb]{0,0,0}\makebox(0,0)[lb]{\smash{$z_{p}^{1}$}}}%
    \put(0.34375546,0.32552386){\color[rgb]{0,0,0}\makebox(0,0)[lb]{\smash{$L_p^{1}$}}}%
    \put(0.42051902,0.26963795){\color[rgb]{0,0,0}\makebox(0,0)[lb]{\smash{$L_{p}^{2}$}}}%
    \put(0.44498314,0.33291799){\color[rgb]{0,0,0}\makebox(0,0)[lb]{\smash{}}}%
  \end{picture}%
\endgroup%